\documentclass[a4paper]{article}
\usepackage{amsthm,amssymb,amsmath,enumerate}
\usepackage{tikz}
\usepackage{hyperref}
\hypersetup{
	colorlinks=true,  
	linkcolor=black,    
	citecolor=black,    
	urlcolor=black      
}

\newtheorem{definition}{Definition}

\newtheorem{theorem}[definition]{Theorem}

\newtheorem{corollary}[definition]{Corollary}
\newtheorem{lemma}[definition]{Lemma}

\renewcommand{\phi}{\varphi}
\renewcommand{\epsilon}{\varepsilon}

\newcommand{\comment}[1]{}
\newcommand{\N}{\mathbb N}
\newcommand{\R}{\mathbb R}

\newcommand{\EP}{Erd\H os-P\'osa property}

\title{Long $A$-$B$-paths have the edge-Erd\H os-P\'osa property}
\author{Matthias Heinlein \and Arthur Ulmer\thanks{Ulm University, {\tt arthur.ulmer@uni-ulm.de}, supported by DFG, grant no.\  BR 5449/1-1}}
\date{}

\begin{document}
\maketitle

\begin{abstract}
For a fixed integer $\ell$ a path is long if its length is at least $\ell$. We prove that for all integers $k$ and $\ell$ there is a number $f(k,\ell)$ such that for every graph $G$ and vertex sets $A,B$ the graph $G$ either contains $k$ edge-disjoint long $A$-$B$-paths or it contains an edge set $F$ of size $|F|\leq f(k,\ell)$ that meets every long $A$-$B$-path. This is the edge analogue of a theorem of Montejano and Neumann-Lara (1984). We also prove a similar result for long $A$-paths and long $\mathcal{S}$-paths. 




\end{abstract}

\section{Introduction}
Menger showed that in any graph with vertex sets $A$ and $B$ and for any $k \in\N$ there are either $k$ disjoint \emph{$A$-$B$-paths} in $G$ or at most $k-1$ vertices that intersect all $A$-$B$-paths (an $A$-$B$-path is a path from $A$ to $B$ that is internally disjoint from $A$ and $B$). 
Montejano and Neumann-Lara \cite{MN84} proved a similar result for \emph{long} $A$-$B$-paths (paths of length at least $\ell$ for some fixed integer $\ell$): for a fixed integer $\ell$, any positive integer $k$ and any graph $G$ with vertex sets $A$ and $B$ there are either $k$ disjoint long $A$-$B$-paths in $G$ or a set of at most $(3(\ell+2)-5)k$ vertices that intersects all long $A$-$B$-paths. Relating to the classic result of Erd\H os and P\'osa \cite{EP65} on the relation between the maximum number of disjoint cycles and the minimum size of a vertex set that intersects all cycles, we say that long $A$-$B$-paths have the \emph{\EP}.



More generally, a family of graphs $\mathcal{H}$ (possibly with some extra structure, e.g. long $A$-$B$-paths) is said to have the \emph{(vertex-)\EP~}if there is a function $f:\N \to \R$ such that for every $k \in \mathbb{N}$ and every graph $G$ there are either $k$ \emph{vertex-disjoint} subgraphs of $G$ that belong to $\mathcal{H}$ or there is a set $X$ of at most $f(k)$ \emph{vertices} in $G$ such that $X$ intersects all subgraphs of $G$ that belong to $\mathcal{H}$. 


By replacing every occurence of "vertex" by "edge" in the definition of the \EP, an edge variant naturally arises. This property is weaker in the sense that we only need to find edge-disjoint subgraphs but at the same time it is stronger since we have to find a set of edges that intersects all these subgraphs. More precisely, a family of graphs $\mathcal{H}$ (possibly with some extra structure) has the \emph{edge-\EP}~if there is a function $f:\N \to \R$ such that for every $k \in \mathbb{N}$ and every graph $G$ there are either $k$ \emph{edge-disjoint} subgraphs of $G$ that belong to $\mathcal{H}$ or there is a set $X$ of at most $f(k)$ \emph{edges} in $G$ such that $X$ intersects all subgraphs of $G$ that belong to $\mathcal{H}$. We call a set of edges that intersects all subgraphs that belong to $\mathcal{H}$ an \emph{edge hitting set} for $\mathcal{H}$ (or mostly just hitting set) and the function $f$ a \emph{hitting set function} for $\mathcal{H}$. 


Long $A$-$B$-paths have the vertex-\EP~but do they have the edge-\EP, too? We prove:

\begin{theorem}
Long $A$-$B$-paths have the edge-\EP.
\end{theorem}

Bruhn, Heinlein and Joos \cite{BHJ18b} showed that long \emph{$A$-paths} also have the vertex-\EP~(paths with both endvertices in a vertex set $A$ and otherwise disjoint of $A$) and asked whether the same remains true for the edge variant. We answer this question:

\begin{theorem}
Long $A$-paths have the edge-\EP.
\end{theorem}


From this results we can easily follow that also long \emph{$\mathcal{S}$-paths} have the property (for a partition $\mathcal{S}=\{A_1,\ldots, A_n\}$ of a set $A$ this is an $A$-path with the endvertices in two different partition sets).

\begin{theorem}
Long $\mathcal{S}$-paths have the edge-\EP.
\end{theorem}

The ordinary \EP~is fairly well studied. The most general result is arguably due to Robertson and Seymour on \emph{$H$-models} \cite{RS86}, graphs that can be contracted to some graph $H$. The set of $H$-models has the \EP~if and only if $H$ is planar, also see \cite{BHJR18} for a recent proof of this with an essentially best possible hitting set function. In contrast, the edge-\EP~is less well understood. In particular, no edge-analogue of the result by Robertson and Seymour is possible. As in the vertex version, the set of $H$-models does not have the edge-\EP~for non-planar graphs \cite{RT16}. But, contrary to the vertex version, the same is true for large (and subcubic) trees and also large ladders \cite{BHJ18c}, which are both planar. Hence only one direction of that equivalence is still true in the edge version. 

There are not many results on the edge-\EP, besides some further small results on $H$-models \cite{BH17,BHJ18,RST13,U17}, it is only known that $A$-$B$-paths, $A$-paths and $\mathcal{S}$-paths have the edge-\EP~but any of these paths such that its length is congruent to some $x$ modulo some $m$ do not have it \cite{BHJ18b}.

For a comprehensive list of \EP~results (vertex and edge version) see \cite{BHJ18b} or \cite{RT16}. In this paper we will use the standard notation of Diestel \cite{diestelBook17}.

\section{Proof of the Main Result}

For $\ell\leq 2$ we start by proving that long $A$-$B$-paths have the edge-\EP. 
We do induction on $\ell$. If $A$-$B$-paths of length at least $\ell$ have the edge-\EP, then also $A$-paths of length at least $\ell-1$ as well as $A^*$-$B$-path of length at least $\ell-1$ ($A$-$B$-paths, that can contain vertices of $A$ in their interior) have the edge-\EP. From this we deduce that $A^*$-$B^*$-paths of length at least $\ell-1$ (paths that can contain vertices of both $A$ and $B$ in their interior) have the edge-\EP~and this implies that also $A$-$B$-paths of length at least $\ell+1$ have the edge-\EP~(also see Figure \ref{diagram}). Hence, by induction, we are finished. This proves that both long $A$-$B$-paths and long $A$-paths have the edge-\EP. Note that we assume $A$ and $B$ to be disjoint, in the next section we deal with the other case.


\begin{figure}
\begin{tikzpicture}
\node at (0,0) [circle,fill=black](AB){};
\node at (2,-1)[circle,fill=black](A){};
\node at (2,1)[circle,fill=black](AsB){};
\node at (6,0)[circle,fill=black](AsBs){};
\node at (8,0)[circle,fill=black](ABl1){};
\node at (-2,0)[align=left](x1){$A$-$B$-paths of\\ at least length $\ell$};
\node at (2.5,1.6)[align=left](x2){$A$-paths of\\ at least length $\ell-1$};
\node at (2.5,-1.6)[align=left](x3){$A^*$-$B$-paths of\\ at least length $\ell-1$};
\node at (6.5,0.8)[align=left](x4){$A^*$-$B^*$-paths of\\ at least length $\ell-1$};
\node at (10,0)[align=left](x5){$A$-$B$-paths of\\ at least length $\ell+1$};

\path[->,thick] (AsBs) edge node {} (ABl1);
\path[->,thick] (AB) edge node {} (A);
\path[->,thick] (AB) edge node {} (AsB);
\path[->,thick] (AsB) edge node {} (AsBs);
\path[->,thick] (A) edge node {} (AsBs);

\end{tikzpicture}
\caption{In this way we can deduce that if $A$-$B$-paths of length at least $\ell$ have the edge-\EP, then also $A$-$B$-paths of length at least $\ell+1$ have it.}
\label{diagram}
\end{figure}
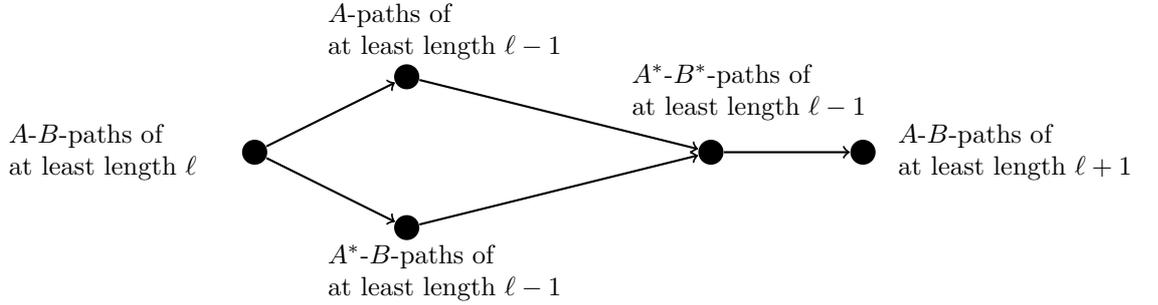

For the start of the induction we need a corollary of the edge-version of Menger's theorem.

\begin{corollary}[Menger]
Let $k \in \N$ and let $G$ be a graph with vertex sets $A$ and $B$. Then there are either $k$ edge-disjoint $A$-$B$-paths or a set of at most $k-1$ edges that intersects all those paths in $G$.
\end{corollary}

\begin{lemma}
Let $k\in\N$ and $\ell \in \{1,2\}$ and let $G$ be a graph with disjoint vertex sets $A$ and $B$. Then there are either $k$ edge-disjoint $A$-$B$-paths of length at least $\ell$ or a set of at most $k-1$ edges that intersects all those paths in $G$.
\end{lemma}

\begin{proof}
First let $\ell=1$. We use Mengers theorem to either find $k$ edge-disjoint $A$-$B$-paths or a set of at most $k-1$ edges that intersects all those paths. Note that since $A$ and $B$ are disjoint every $A$-$B$-path has length at least $1$. So we are done.

Now let $\ell=2$. We can delete all edges between vertices of $A$ and $B$ since they cannot be part of any $A$-$B$-path of length at least $2$. Since these are the only $A$-$B$-paths of length $1$, all remaining paths have  length at least $2$. Again using Menger's theorem we are done.
\end{proof}

Let $f(k, \ell)$ be the size of a hitting set for $A$-$B$-paths of length at least $\ell$ if a graph does not contain $k$ edge-disjoint $A$-$B$-paths of length at least $\ell$, in other words let $f(\cdot,\ell)$ be the hitting set function for long $A$-$B$-paths. As we have seen we can set $f(k,1)=f(k,2)=k-1$. It is possible to compute $f(k, \ell)$ inductively using only values of $f$ that have been computed before but as it is quite convoluted we will just give the recursive formula later on.

Now that we have proven the start of the induction, our inductive hypothesis is:

\begin{equation}
	\begin{minipage}[c]{0.8\textwidth}\em
		In any graph $G$ there are either $k$ edge-disjoint $A$-$B$-paths of length at least $n$ or a set of at most $f(k,n)$ edges that intersects all those paths for all $n\in \{1,\ldots,\ell\}$.
	\end{minipage}\ignorespacesafterend 
	\label{indhypo}
\end{equation}

Before we continue we want to talk about a common trick in Erd\H os-P\'osa-questions that lets us assume that there are no \emph{small subgraphs}. For that let $\mathcal{H}$ again be some class of graphs (possibly with some extra structure) and let $c_k$ be a constant for $k\in\N$. To prove that $\mathcal{H}$ has the edge-\EP~we need to show that for every interger $k$ and any graph $G$ there are either $k$ edge-disjoint subgraphs in $G$ that belong to $\mathcal{H}$ or a set of edges of at most size $f(k)$ that intersects all such subgraphs for some function $f$. For $k=1$ this statement is always true, so we can do induction on $k$. Let $k\geq 2$, we can assume that:

\begin{equation}
	\begin{minipage}[c]{0.8\textwidth}\em
		There is no subgraph of $G$ that belongs to $\mathcal{H}$ with at most $c_k$ edges.
	\end{minipage}\ignorespacesafterend 
	 \label{nosmall}
\end{equation}
 
Suppose there is such a subgraph $H_1$ in $G$. Remove all edges of $H_1$ in $G$ and use induction. If we find $k-1$ edge-disjoint subgraphs $H_2,\ldots, H_k$ in $G-E(H_1)$ that belong to $\mathcal{H}$, then these are clearly edge-disjoint from $H_1$. Thus, we get $k$ edge-disjoint subgraphs in $G$ that belong to $\mathcal{H}$. If we find a hitting set in $G-E(H_1)$ of at most size $f(k-1)$, then we add the $c_k$ edges of $H_1$ to the hitting set. By this we get a hitting set for $\mathcal{H}$ in $G$, whose size is bounded by $f(k-1)+c_k$. Since $c_k$ is a constant we can simply choose $f$ such that $f(k)\geq f(k-1)+c_k$. Therefore there is no need to look at graphs with such a small subgraph.

Again we can easily extend this result to long $A$-$B$-paths (or anything else). Since in Erd\H os-P\'osa problems the case $k=1$ is always true, we do not need to check if we can apply induction. In the following we sometimes do not explicitly state that we do induction but we still use (\ref{nosmall}).

In a graph $G$ with a vertex set $A$, an $A$-path is a path with both endvertices in $A$ and otherwise disjoint from $A$ (length at least $1$, single vertices of $A$ are not considered as $A$-paths). We prove that long $A$-paths have the edge-\EP. Note that the meaning of "long" changes depending on the length of the paths in the statement that we want to prove.

\begin{theorem}\label{A}

$A$-paths of length at least $\ell-1$ have the edge-\EP. Let $g(\cdot, \ell-1)$ be the hitting set function for these paths, then $$g(k, \ell-1)=\max\{g(k-1,\ell-1)+2(\ell-1), f(k,\ell-1)+2g(k-1,\ell-1), 3f(k,\ell-1)\}.$$
\end{theorem}

\begin{proof}
We may assume that $k\ge 2$ and also that $G$ does not contain $k$ edge-disjoint long $A$-paths. Furthermore, by (\ref{nosmall}), there are no long $A$-paths of at most length $2(\ell-1)$.
		By a bipartition of $A$ we mean a tuple $(A_1,A_2)$ of two disjoint nonempty subsets $A_1,A_2$ of $A$
		such that $A_1\cup A_2=A$.
		
		If there is a bipartition $(A_1,A_2)$ of $A$ such that there are $k$ edge-disjoint long $A_1$--$A_2$-paths they form a set of $k$ edge-disjoint long $A$-paths, which contradicts our assumption in the beginning.
		Hence, by (\ref{indhypo}), for every bipartition $(A_1,A_2)$ there is a set $X(A_1,A_2)$ of at most $f(k, \ell-1)$ edges 
		such that $G-X(A_1,A_2)$ contains no long $A_1$--$A_2$-path.
		Therefore, every long $A$-path in $G-X(A_1,A_2)$ has both ends in $A_1$ or both ends in $A_2$.
		For $i=1,2$ define graphs $G_i$ as
		\[
		G_i=G-X(A_1,A_2)-A_{3-i}.
		\]
		We observe that 
		\begin{equation}\label{eq:leftRightPaths}
		\begin{minipage}[c]{0.8\textwidth}\em
		if $P_1$ is a long $A_1$-path in $G_1$ and $P_2$ is a long $A_2$-path in $G_2$,
		then $P_1$ and $P_2$ are disjoint.
		\end{minipage}\ignorespacesafterend 
		\end{equation} 
		Otherwise follow $P_1$ from one of its endvertices $u$ until we first meet $P_2$ in a vertex $v$.
		Then, since each long $A$-path has length at least $2(\ell-1)$, one endvertex $w$ of $P_2$ has at least distance $\ell-1$ from $v$. So follow $P_1$ from $u$ to $v$ and then $P_2$ to $w$. By this, 
		we have obtained a long $A_1$-$A_2$-path in $G_1\cup G_2=G-X(A_1,A_2)$,
		which is impossible.
		This proves \eqref{eq:leftRightPaths}.
		
		Consider first the case that there is a bipartition $(A_1,A_2)$ of $A$ such that for both $i=1,2$
		there is a long $A_i$-path $P_i$ in $G_i$.
		If there were $k-1$ edge-disjoint $A_i$-paths in $G_i$ for either $i\in \{1,2\}$, then together with $P_{3-i}$ in $G_{3-i}$ we find $k$ edge-disjoint ones. Note that we use \eqref{eq:leftRightPaths} here.
		So by induction we can assume to find hitting sets $X_i$ for long $A_i$-paths in $G_i$ such that $|X_i|\le g(k-1,\ell-1)$ for $i\in\{1,2\}$.
		Then, $X=X(A_1,A_2)\cup X_1\cup X_2$ is a hitting set for long $A$-paths in $G$ of at most size $f(k,\ell-1)+2g(k-1,\ell-1)$.
		
		Next, if there is a bipartition $(A_1,A_2)$ of $A$ such that $G_i$ 
		contains no long $A_i$-path for both $i\in \{1,2\}$, then
		$X(A_1,A_2)$ is a hitting set for long $A$-paths of at most size $f(k,\ell-1)$.

		Summing up, we may assume that for every bipartition $(A_1,A_2)$ of $A$
		either 
		\begin{equation}\label{eq:remains}
		\begin{minipage}[c]{0.8\textwidth}\em
		$G_1$ contains at least one long $A_1$-path and $G_2$ contains no long $A_2$-path
		\end{minipage}\ignorespacesafterend 
		\end{equation} 
		or the other way round. In the following, we always think of $G_1$ as the subgraph with a long $A_1$-path.
		
		Among all such bipartitions choose $(A_1,A_2)$ such that $A_1$ has minimal size.
		Because $G_1$ contains a long $A_1$-path, $A_1$ consists of at least two vertices.
		Consider an arbitrary bipartition $(A_3,A_4)$ of $A_1$.
		We will show that $X=X(A_3,A_2\cup A_4)\cup X(A_4,A_2\cup A_3)\cup X(A_1,A_2)$ is a hitting set for long $A$-paths.
		
		Because $(A_3,A_2\cup A_4)$ is a bipartition and $|A_3|<|A_1|$, 
		the set $X(A_3,A_2\cup A_4)$ meets every long $A$-path with at least one endvertex in $A_3$.
		By symmetry between $A_3$ and $A_4$, the set $X(A_4,A_2\cup A_3)$ meets every long $A$-path with at least one endvertex in $A_4$.
		Similarly, by the definition of $X(\cdot,\cdot)$ and \eqref{eq:remains}, 
		the set $X(A_1,A_2)$ meets every long $A$-path with at least one endvertex in $A_2$.
		Hence, the union $X$ of these three sets meets every directed $A$-path in $G$ as $A$ is the union of $A_2,A_3$ and $A_4$.
		The size of $X$ is bounded by $3f(k,\ell-1)$ and hence we are done.
\end{proof}

In $A$-$B$-paths we explicitly forbid vertices of $A$ or $B$ to be in the interior of the path. But for our proof we also need to look at paths which are allowed to do just that. For such paths we mark the set which can be used in the interior with a star. So an \emph{$A^*$-$B$-path} is a path which starts in $A$, ends in $B$, and such that vertices of $A$ can be in its interior but not of $B$. An \emph{$A^*$-$B^*$-path} is a path which starts in $A$ and ends in $B$ and does not have any further restrictions.

\begin{lemma}\label{A*B}
If $A$ and $B$ are disjoint, then $A^*$-$B$-paths of length at least $\ell-1$ have the edge-\EP. Let $f_1(\cdot,\ell-1)$ be the hitting set function for these paths, then $$f_1(k,\ell-1)=f(k,\ell).$$
\end{lemma}

\begin{proof}
Let $k \in\N$ and let $G$ be a graph and $A$ and $B$ disjoint subsets of its vertices.
For each vertex $a$ in $A$ we add as many new vertices to a set $C$ as there are edges incident to $a$, and make each new vertex adjacent to only $a$. Let $G'$ be this new graph.

Now assume that there is a $C$-$B$-path $P'$ of length at least $\ell$. We know that there is exactly one vertex of $C$ and one of $B$ in $P'$ and both of them are endvertices. Let $c$ be the endvertex in $C$ then the vertex after $c$ on $P'$ lies in $A$ by construction of $C$. If we let $P=P'-c$ be the remaining path after removing $c$, then $P$ starts in $A$, ends in $B$, no interior vertex is part of $B$, is disjoint of $C$ and therefore contained in $G$, and also has length at least $\ell-1$. So $P'$ is an $A^*$-$B$-path of length at least $\ell-1$ in $G$. Hence, if we find $k$ edge-disjoint $C$-$B$-paths of length at least $\ell$ we find $k$ edge-disjoint $A^*$-$B$-paths of length at least $\ell-1$ which are still edge-disjoint. Thus we are done.

Hence, by (\ref{indhypo}), we can assume to find a hitting set of at most size $f(k,\ell)$ for all $C$-$B$-paths of length at least $\ell$. Let $X'$ be a minimum sized hitting set.

If there is an edge $ac$ in $X'$ for $a \in A$ and $c \in C$ then all edges between $a$ and $C$ are in the hitting set. Assume the contrary. Hence, there is another edge $ac'$ for $c' \in C$ which is not in $X'$. Since $X'$ was chosen minimum, there should be a $C$-$B$-path of length at least $\ell$ in $G-(X'\setminus\{ac\})$ which has to use $ac$. Now replace $ac$ by $ac'$ and we get a $C$-$B$-path of length at least $\ell$ that is not hit by $X'$ which is a contradiction. 

We construct a hitting set $X$ in $G$ by uniting all edges of $X' \cap G$ and also all edges incident to vertices $a\in A$ in $G$ such that there is an edge $ac$ in $X'$ for some $c \in C$. By the observation before and since we have exactly as many edges from each vertex $a\in A$ to $C$ as it has degree in $G$ the size of $X$ is at most the size of $X'$.

Assume there is still an $A^*$-$B$-path $P$ of length at least $\ell-1$ in $G-X$. Let $a \in A$ be one of its endvertices, then $a$ has at least one neighbour $c \in C$ in $G'$. If the edge $ac$ had been in $X'$ then we should have deleted all incident edges of $a$ in $G$ and therefore $P$ couldn't have existed. Otherwise,  the path $P \cup ac$ is a long $C$-$B$-path in $G'$ that avoids $X'$ which is a contradiction. So we are done.

\end{proof}

Let $P$ be a path and let $x,y$ be vertices on this path. For the subpath of $P$ from $x$ to $y$ we write $xPy$.

\begin{lemma} \label{starstar}

$A^*$-$B^*$-paths of length at least $\ell-1$ have the edge-\EP. Let $f_2(\cdot,\ell-1)$ be the hitting set function for these paths, then $$f_2(k,\ell-1)=\max\{4f_1(k,\ell-1)+g(k,\ell-1),\ell(\ell-2)+f_2(k-1,\ell-1)\}.$$


\end{lemma}

\begin{proof}
By (\ref{nosmall}), we may assume that there is no long $A^*$-$B^*$-path of at most length $\ell(\ell-2)$. 

We first look for long $(A\setminus B)^*$-$B$-paths (note that $(A\setminus B)$ and $B$ are disjoint). Each such path is also a long $A^*$-$B^*$-path and, hence, if there are $k$ edge-disjoint long $A^*$-$B$-paths, we are finished. By Lemma \ref{A*B}, we can assume that there is a set $X_1$ of at most size $f_1(k,\ell-1)$ edges intersecting all those paths. Likewise we find sets $X_2,X_3,X_4$ each of at most size $f_1(k,\ell-1)$ and intersecting all $(B\setminus A)^*$-$A$-paths, $A^*$-$(B\setminus A)$, and $B^*$-$(A\setminus B)$-paths respectively. By the same argument and Theorem \ref{A} we find a hitting set $X_5$ of at most size $g(k,\ell-1)$ for $(A \cap B)$-paths. Let $X = \cup_{i=1}^5 X_i$, then the size of $X$ is bounded by $4f_1(k,\ell-1)+g(k,\ell-1)$.

Suppose there is a long $A^*$-$B^*$-path $P$ left after removing $X$ from the graph. We follow $P$ from its first endvertex $a_1$ in $A$ until we come across a vertex $b_1$ in $B$ (a different vertex if $a_1\in A \cap B$). First assume that $a_1 \in A\setminus B$, then this subpath $a_1Pb_1$ is an $(A \setminus B)^*$-$B$-path and hence has to be short. Likewise, if $a_1 \in A \cap B$ and $b_1 \in B \setminus A$ then this is an $A^*$-$(B\setminus A)$-path and therefore has to be short. And lastly if $a_1 \in A\cap B$ and $b_1 \in A\cap B$ then this is an $(A\cap B)$-path and thus is short. Hence, $a_1Pb_1$ has at most length $\ell-2$ and because of that the path $P$ has to continue after $b_1$. From $b_1$ we do the same until we get to the first vertex $a_2$ of $A$. Analogously $b_1Pa_2$ is short, i.e. has at most length $\ell-2$. Again $P$ has to continue after $a_2$ as the length of $a_1Pa_2$ is at most $2(\ell-2)$ and we assumed that no long $A^*$-$B^*$-path of at most length $\ell(\ell-2)$ exists. We do this $\ell -1$ times and if $P$ does not end in $B$ at that point, we follow $P$ one last time until a vertex of $B$ comes up. So at most $\ell$ times. We get at least $\ell-1$ subpaths $a_iPb_i$ and $b_iPa_{i+1}$, each of which has length at least $1$ but at most length $\ell-2$. Hence $a_1Pb_{\lceil \frac{\ell-1}{2} \rceil}$ is an $A^*$-$B^*$-path of length at least $\ell-1$ but at most length $\ell(\ell-2)$, which is a contradiction. Therefore, $X$ is a hitting set of bounded size and we are done.

\end{proof}

\begin{theorem}


If $A$ and $B$ are disjoint, then $A$-$B$-paths of length at least $\ell+1$ have the edge-\EP. For the hitting set function $f(\cdot, \ell+1)$ it holds that $$f(k,\ell+1)=\max\{f_2(2k(2\ell+5)(k-1),\ell-1),f(k-1,\ell)+(2\ell+5)k\}.$$
\end{theorem}

\begin{proof}
We do induction on $k$. For $k=1$ the statement is obviously true, so let $k \geq 2$. 
We can assume that there are no edges between $A$ and $B$, $A$ and $A$, and $B$ and $B$ (for $\ell \geq 2$) since they are not part of any long $A$-$B$-path. If $A$ or $B$ have no neighbours left, then we are done. So if we let $A_1$ be the set of neighbours of $A$ and $B_1$ the set of neighbours of $B$, then $A_1$ and $B_1$ are non-empty and disjoint from both $A$ and $B$.

Now there are either $2k(2\ell+5)(k-1)$ many edge-disjoint $A_1^*$-$B_1^*$-paths of length at least $\ell -1$ in $G-(A \cup B)$ or a set $X$ of size $f_2(2k(2\ell+5)(k-1),\ell -1)$ that intersects all those paths, because of Lemma \ref{starstar}. Note that in the second case $X$ is a hitting set for all $A$-$B$-paths of length at least $\ell+1$ because every such $A$-$B$-path contains an $A_1^*$-$B_1^*$-path in $G-(A\cup B)$ of length at least $\ell-1$ as subpath. 

So let $\mathcal{Q}$ be the set of these paths. Note that each path in $\mathcal{Q}$ can be extended to a long $A$-$B$-path, because each such path lives in $G-(A\cup B)$, has one endvertex in $A_1$, which has a neighbour in $A$, and one endvertex in $B_1$, which has a neighbour in $B$. We claim:


\begin{equation}\label{manypaths}
		\begin{minipage}[c]{0.8\textwidth}\em
		There is one vertex $v \in A_1 \cup B_1$ in which at least $(2\ell+5)(k-1)$ many paths of $\mathcal{Q}$ end.
		\end{minipage}\ignorespacesafterend 
		\end{equation}

We construct a graph $G'$ with vertex set $A_1 \cup B_1$. Connect two vertices in $G'$ if there is a path in $\mathcal{Q}$ whose endvertices are these vertices. Assume there is a matching of size $k$ in $G'$. Every such edge corresponds to a path in $\mathcal{Q}$ and as we have remarked can be extended to a long $A$-$B$-path by adding edges from $A_1$ to $A$ and from $B_1$ to $B$. Since all the endvertices of the edges in the matching are different also the endvertices of the paths are different. Therefore, also the edges that are used to extend the paths are different. So we get $k$ edge-disjoint long $A$-$B$-paths and are finished.

So we can assume that there is a vertex cover of at most size $2k$ in this graph. Hence, by the pigeon hole principle, we know that at least $(2\ell+5)(k-1)$ paths of $\mathcal{Q}$ end in one vertex $v \in A_1 \cup B_1$. This proves the claim.

%

Now remove all other paths from $\mathcal{Q}$. Each path has one endvertex in $A_1$ and one in $B_1$. From each $A_1$-endvertex take one edge to $A$ and from each $B_1$-endvertex take one edge to $B$ and put them into a set $X_1$. The size of $X_1$ is at most $2+(2\ell+5)(k-1)$.

By induction there are either $k-1$ edge-disjoint long $A$-$B$-paths in $G-X_1$ or a hitting set $X$ of at most size $f(k-1, \ell)$ for all of these paths. In the second case clearly $X \cup X_1$ is a hitting set in $G$ of at most size $f(k-1,\ell)+(2\ell+5)k$. So we can assume to find $k-1$ long $A$-$B$-paths $P_1, \ldots, P_{k-1}$.  We choose them such that 

$$\sum_{i=1}^{k-1} |\{\text{non-trivial components of}~ P_i \cap Q : Q \in \mathcal{Q}\}|$$

is minimum. We say a component is non-trivial if it contains at least two vertices (and thus an edge). We claim:

\begin{equation}\label{1edisjpath}
		\begin{minipage}[c]{0.8\textwidth}\em
		There is a path in $\mathcal{Q}$ that is edge-disjoint from $P_1, \ldots, P_{k-1}$.
		\end{minipage}\ignorespacesafterend 
		\end{equation}

First of all we want to show that this would suffice in order to prove the theorem.
If there is a path in $\mathcal{Q}$ that is edge-disjoint from $P_1, \ldots, P_{k-1}$, then we can just add the two edges from its endvertices to $A$ and $B$ respectively that we stored in $X_1$ and we get a long $A$-$B$-path. Since $P_1, \ldots, P_{k-1}$ live in $G-X_1$ they are edge-disjoint from this new path and hence we find $k$ edge-disjoint long $A$-$B$-paths.

To prove this claim, assume that all paths in $\mathcal{Q}$ have at least one edge in common with at least one path $P_j$.
We try to reach a contradiction. On each path $Q \in \mathcal{Q}$ we find one edge of a path $P_j$ that is closest to $v$. By the pigeon hole principle we can find one path $P_i$ that is responsible for at least $(2l+5)$ of these closest edges. Enumerate the paths of $\mathcal{Q}$ on which $P_i$ is the closest path according to the occurence of the closest edges (starting in the endvertex of $P_i$ in $B$). We get the paths $Q_1, \ldots, Q_m$ ($m \geq 2l+5$) on which $P_i$ uses the closest edge.

Let $e$ be the closest edge on $Q_{\ell+2}$ and let $s$ be the endvertex of $e$ that is closer to $v$ on $Q_{\ell+2}$. First assume that $P_i$ and $sQ_{\ell+2}v\setminus\{s\}$ are not disjoint. Starting in $s$ let $x$ be the first intersection on $Q_{\ell+2}$, clearly the path $sQ_{\ell+2}x$ has length at least $1$. We assume that $x$ comes before $e$ on $P_i$ if we start in $B$ (the other case can be handled analogously). Replace the subpath $sP_ix$ by $sQ_{\ell+2}x$ and let $P'$ be this new path. It is easy to check that $P'$ is indeed a path and the length of it is again at least $\ell+1$ because after $s$ we find at least $\ell+1$ last edges (on $Q_{\ell+3},\ldots,Q_m$). 

If $P_i$ comes across $x$ before it uses an  edge of the component containing the closest edge on $Q_{\ell+1}$ we definitely lose this non-trivial component of $P_i \cap Q_{\ell+1}$ (because we skip it). Since $sQ_{\ell+2}v$ is edge-disjoint from all paths $P_j$ ($e$ is the closest edge on $Q_{\ell+2}$), the path $P'$ is still edge-disjoint from all other paths $P_j$. Hence we have found a better choice for the $k-1$ long $A$-$B$-paths, which is a contradiction. Note that by adding $sQ_{\ell+2}x$ we did not create any new components but enlarged the component containing the last edge on $Q_{\ell+2}$.

Now assume $x$ comes after $P_i$ uses an edge of the component containing the closest edge on $Q_{\ell+1}$. The vertex $x$ has at least distance one from $s$ on $Q_{\ell+2}$ and is closer to $v$. We connected $x$ and $s$ on $Q_{\ell+2}$ and therefore the closest edge on $Q_{\ell+2}$ moved closer to $v$. The amount of components that intersect $P_i$ did not increase by the same argument as before (if anything it decreased) and the path $P'$ still contains the closest edges on the paths $Q_1, \ldots, Q_m$. We keep doing this until $sQ_{\ell+2}v$ is disjoint from $v$ and we are unable to move the closest on $Q_{\ell+2}$ closer to $v$ (or until we reach a contradiction). 

So we can assume that $sQ_{\ell+2}v\setminus\{s\}$ is disjoint from $P_i$. In the same way we can show that the same is also true for $Q_{\ell+4}$, i.e. if $e'$ is the closest edge on $Q_{\ell+4}$ and $s'$ the endvertex of $e'$ that is closer to $v$ on $Q_{\ell+4}$, then $s'Q_{\ell+4}v$ is disjoint from $P_i$. Now we just remove the subpath of $P_i$ between $s$ and $s'$ and connect them via the paths $sQ_{\ell+2}v$ and $s'Q_{\ell+4}v$. Note that $sQ_{\ell+2}v$ and $s'Q_{\ell+4}v$ may intersect before $v$ but we definitely can find a subpath that connects $s$ and $s'$. By construction we get a path, whose length is at least $\ell+1$ because it is still responsible for at least $\ell+1$ closest edges.
Moreover, this new  path is edge-disjoint from all other paths $P_j$ since $sQ_{\ell+2}v$ and $s'Q_{\ell+4}v$ are (again $e$ and $e'$ are the closest edges). As we have lost the closest edge on $Q_{\ell+3}$ we also lost the component containing this edge on $Q_{\ell+3}\cap P_i$ and hence found a better choice for the paths $P_1,\ldots , P_{k-1}$. This is a contradiction.
Hence we are finished.
\end{proof}

\section{Corollaries}
Using the results of the previous section we can prove some further results.
First of all we want to show that long $A$-$B$-paths for non-disjoint sets $A$ and $B$ also have the edge-\EP.

\begin{theorem}
$A$-$B$-paths of length at least $\ell$ have the edge-\EP.
\end{theorem}

\begin{proof}
For a given $k$ we may assume that there are neither $k$ edge-disjoint long $A$-$(B\setminus A)$-paths nor $B$-$(A\setminus B)$-paths nor $(A\cap B)$-paths, since these are also long $A$-$B$-paths. Hence by the results of the previous section we find hitting sets $X_1,X_2$ and $X_3$ for these paths of bounded size. It is easy to see that $X_1\cup X_2 \cup X_3$ is a hitting set for all long $A$-$B$-paths of bounded size.
\end{proof}





A more general type of $A$-paths are $\mathcal{S}$-paths. Let $A$ be a set of vertices and $\mathcal{S}=\{A_1,\ldots,A_n\}$ a partition of $A$, i.e. all $A_i$ are pairwise disjoint and their union is $A$.
An $\mathcal{S}$-path is a path that starts in some partition set $A_i$, ends in another set $A_j$ and is otherwise disjoint of $A$. We want to show that long $\mathcal{S}$-paths have the edge-\EP. 

\begin{theorem}
$\mathcal{S}$-paths of length at least $\ell$ have the edge-\EP.
\end{theorem}

\begin{proof}
Let $\mathcal{S}=\{A_1, \ldots, A_n\}$ be a partition of a vertex set $A$ in a graph $G$ and let $\ell\geq 2$. First of all we can remove all edges between vertices of $A$ since they are not part of any long $\mathcal{S}$-path anyway. Now subdivide each edge that is incident to a vertex of $A$ exactly once and contract all sets $A_i$ to a single vertex $a_i$ ($a_i$ is adjacent to all neighbours of $A_i$). Let $G'$ be this new graph and $A'=\{a_1,\ldots,a_n\}$. Note that we subdivide the edges only to avoid having to deal with multi-edges. Also each edge in $G'$ corresponds to an edge in $G$ after the subdivision.

Assume there is an $A'$-path of length at least $\ell+2$ in $G'$ between $a_i$ and $a_j$. By construction it has to contain the subdivision of exactly two edges that were adjacent to $A$ and hence if we look at the corresponding path in $G$ (contract the subdivided edges, decontract the vertex sets $A_i,A_j$), then we find a path of length at least $\ell$ between $A_i$ and $A_j$ which is also internally disjoint of $A$ and therefore an $\mathcal{S}$-path. Thus if we find $k$ edge-disjoint $A'$-paths of at least lenght $\ell+2$, we find $k$ edge-disjoint $\mathcal{S}$-paths of length at least $\ell$.

Since $A'$-paths of length at least $\ell+2$ have the edge-\EP, we can assume to find a hitting set $X'$ for these paths of bounded size in $G'$. As we remarked, each edge in $G'$ corresponds to an edge in $G$ after the subdivision of some edges. Let $X$ contain all edges of $X'\cap E(G)$ and also for each edge in $X'$ such that the corresponding edge is part of a subdivision we add the edge of $G$ that was subdivided. It is easy to check that $G-X$ contains no long $\mathcal{S}$-path. Furthermore, the size of $X$ is bounded by the size of $X'$. Therefore, we are finished.
\end{proof}

We can prove that long $A^*$-paths have the edge-\EP. The proof is also found in \cite{U17}.

\begin{corollary}
$A^*$-paths of length at least $\ell$ have the edge-\EP.
\end{corollary}

\begin{proof}
By (\ref{nosmall}), we may assume that there is no long $A^*$-path of at most length $2\ell-2$.
We claim that:

\begin{enumerate}
\item
There are $k$ edge-disjoint long $A$-paths if and only if there are $k$ edge-disjoint long $A^*$-paths.
\item
Any hitting set for long $A$-paths is a hitting set for long $A^*$-paths and also vice versa.
\end{enumerate}

Take any long $A^*$-path which has a vertex $a \in A$ in its interior. We take the two $A^*$-paths that arise when we split the path at $a$. If both of those have at most length $\ell-1$, then the whole path has at most length $2\ell-2$, which is a contradiction. So at least one path has length at least $\ell$. Using this subpath we can do the same again until we find a long $A^*$-path that is also an $A$-path as a subpath of the original path.

Now in the first part the implication $"\Rightarrow"$ is trivial since every $A$-path is an $A^*$-path. The other direction follows by the statement above. If we have $k$ edge-disjoint long $A^*$-paths, then each of them either already is a long $A$-path or contains a subpath that is a long $A$-path. Thus the claim follows.

In the second part the implication $"\Leftarrow"$ is trivial, i.e. any hitting set for long $A^*$-paths is a hitting set for long $A$-paths. As above the other direction follows by the observation before. Assume we have a hitting set for all long $A$-paths, remove it from the graph. If there was still a long $A^*$-path then this path would contain a long $A$-path as subpath which would be a contradiction.

Now one can see that since $A$-paths have the edge-\EP~ also $A^*$-paths have it.
\end{proof}

Using the fact that long $A$-paths have the edge-\EP, we can also prove that long cycles have this property. This has already been proven in \cite{BHJ18c}, but we give a shorter proof (although with a worse hitting set function). A \emph{$1$-vertex-hitting-set} is a vertex hitting set of size $1$. We need the following observation.

\begin{lemma}[Bruhn, Heinlein \cite{BH17}]
If a family of graphs $\mathcal{H}$ has the vertex-\EP, then if every graph with a $1$-vertex-hitting-set for $\mathcal{H}$ either has $k$ edge-disjoint subgraphs that all belong to $\mathcal{H}$ or a hitting set of edges of bounded size, then $\mathcal{H}$ also has the edge-\EP.
\end{lemma}

Essentially what this means is that if some family, that has the vertex-\EP, has the edge-\EP~in the class of graphs with a $1$-vertex-hitting-set, then  it already has it in all graphs. It has already been shown that long cycles have the vertex-\EP~\cite{BBR07}, so we can use this trick  to show that long cycles also have the edge-\EP.

\begin{corollary}
Cycles of length at least $\ell$ have the edge-\EP.
\end{corollary}

\begin{proof}
Let $G$ be a graph with a $1$-vertex-hitting-set, let $x$ be the vertex that intersects all long cycles
and let $N(x)$ be the set of neighbours of $x$. Remove $x$ from $G$ and for each vertex $v \in N(x)$ add a new vertex that is adjacent only to $v$ and let $A$ be the set of these new vertices. Let $G'$ be this graph.

Assume there are $k$ edge-disjoint $A$-paths of length at least $\ell$ in $G'$. In all these paths we replace the edges from $A$ to $v\in N(x)$ by the edge $xv$. Since we have exactly two such edges in every path and since they meet in $x$, we get a cycle in $G$ which is also long. All these cycles are edge-disjoint since otherwise an edge between a vertex of $N(x)$ and $A$ would have been used twice, which is impossible.

Now assume there is a hitting set $X'$ for all $A$-paths in $G'$. We adapt this to a hitting set $X$ in $G$ by again replacing all edges from a vertex $v\in N(x)$ to $A$ in $X'$ by the edge $vx$ and leaving the rest as it is. Assume there is still a long cycle $C$ in $G-X$. This cycle has to intersect $x$ since $x$ is a $1$-vertex-hitting-set in $G$. Now we replace the two edges in $C$ from $x$ to $v_1,v_2 \in N(x)$ by the edges from $v_1$ and $v_2$ to $A$. By this replacement we get a long $A$-path in $G'$. This path avoids $X'$ as otherwise an edge from $A$ to $v_i$ would be contained in $X'$, but then also the edge $xv_i$ would be in $X$ and then $C$ would not be a subgraph of $G-X$. Hence we are done. 
\end{proof}

\bibliographystyle{amsplain}
\bibliography{erdosposa}

\end{document}